\documentclass[11pt]{amsart}
   
\usepackage{amsthm,amsmath,amssymb,amscd,amsfonts,latexsym}

\theoremstyle{plain}
\newtheorem{theorem}{Theorem}[section]

\newtheorem{proposition}[theorem]{Proposition}

\newtheorem{corollary}[theorem]{Corollary}

\newtheorem{def-thm}[theorem]{Definition-Theorem}

\theoremstyle{definition}

\newtheorem{remark}[theorem]{Remark}

\newtheorem*{acknowledgement}{Acknowledgement}

\newcommand{\PP}{\mathbb{P}}

\newcommand{\RR}{\mathbb{R}}

\newcommand{\CC}{\mathbb{C}}

\newcommand{\OO}{{\mathcal O}}
\newcommand{\VV}{{\mathcal V}}

\newcommand{\rr}{{r^+_M}}

\DeclareMathOperator{\spann}{span}
\DeclareMathOperator{\Hom}{Hom}

\DeclareMathOperator{\tr}{tr}

\DeclareMathOperator{\PEff}{PEff}

\allowdisplaybreaks

\begin{document}

\title[On projective K\"ahler manifolds of partially positive curvature]{On projective K\"ahler manifolds of partially positive curvature and rational connectedness}

\begin{abstract} 
In a previous paper, we proved that a projective K\"ahler manifold of positive total scalar curvature is uniruled. At the other end of the spectrum, it is a well-known theorem of Campana and Koll\'ar-Miyaoka-Mori that a projective K\"ahler manifold of positive Ricci curvature is rationally connected. In the present work, we investigate the intermediate notion of $k$-positive Ricci curvature and prove that for a projective $n$-dimensional K\"ahler manifold of $k$-positive Ricci curvature the MRC fibration has generic fibers of dimension at least $n-k+1$. We also establish an analogous result for projective K\"ahler manifolds of semi-positive holomorphic sectional curvature based on an invariant which records the largest codimension of maximal subspaces in the tangent spaces on which the holomorphic sectional curvature vanishes. In particular, the latter result confirms a conjecture of S.-T. Yau in the projective case.
\end{abstract}

\author{Gordon Heier} \author{Bun Wong}

\address{Department of Mathematics\\University of Houston\\4800 Calhoun Road, Houston, TX 77204\\USA}
\email{heier@math.uh.edu}

\address{Department of Mathematics\\UC Riverside\\900 University Avenue\\ Riverside, CA 92521\\USA}
\email{wong@math.ucr.edu}


\thanks{The first author is partially supported by the National Security Agency under Grant Number H98230-12-1-0235. The United States Government is authorized to reproduce and distribute reprints notwithstanding any copyright notation herein.}

\maketitle

\section{Introduction and statement of the results} 
In our previous work \cite{heier_wong_cag}, we showed that an $n$-dimensional complex projective K\"ahler manifold $M$ of positive total scalar curvature is {\it uniruled}, which means that there exists a dominant rational map from $\PP^1\times N$ onto $M$, where $N$ is a complex projective variety of dimension $n-1$. Recall that being uniruled is the same as having a rational curve passing through every point. Our proof was based on the important equivalence, established by Boucksom-Demailly-Peternell-P\u aun in \cite{BDPP_13}, of uniruledness to the property that the canonical line bundle of $M$ is not pseudo-effective.\par
On the other hand, it is a well-known result of Campana \cite{campana_con_rat} and Koll\'ar-Miyaoka-Mori \cite{k_m_m_jag} that a projective K\"ahler manifold of positive Ricci curvature (aka a Fano manifold) is {\it rationally connected}, i.e., any two points can be joined by a chain of rational curves or, equivalently, by a single rational curve. In light of all of this, we establish two main theorems in the case of partially positive curvature which are natural extensions of the previous results. Throughout the paper, we shall work over the field of complex numbers $\CC$.\par

The main kinds of partially positive curvatures which we will consider are $k$-positive Ricci curvature and semi-positive holomorphic sectional curvature with a certain amount of positivity captured by the numerical invariant $r_M^+$. For the definitions, we refer to Section \ref{diff_alg_geom_backgr}. Our first main theorem is the following.

\begin{theorem}\label{mthm_k_ricci} Let $M$ be a projective manifold with a K\"ahler metric with $k$-positive Ricci curvature, where $k\in \{1,\ldots,n:=\dim M\}$. Then a generic fiber of the MRC fibration of $M$ has dimension at least $n-k+1$.
\end{theorem}

\begin{remark}
The case $k=1$ in Theorem \ref{mthm_k_ricci} represents the above-mentioned well-known result of Campana and Koll\'ar-Miyaoka-Mori. The case $k=n$ is the case of positive scalar curvature, which was handled in \cite{heier_wong_cag} (under the even weaker assumption of positive {\it total} scalar curvature).
\end{remark}

We remark that  at the end of the proof of Theorem \ref{mthm_k_ricci}, the symbol $\varepsilon$ merely has to denote a semi-positive continuous function which is positive at at least one point in order for the proof to go through verbatim. Thus, Theorem \ref{mthm_k_ricci} can immediately be generalized to the following theorem.

\begin{theorem}\label{mthm_semi_k_ricci} Let $M$ be a projective manifold with a K\"ahler metric with $k$-semi-positive Ricci curvature, where $k\in \{1,\ldots,n:=\dim M\}$. Assume that that there exists at least one point of $M$ at which the K\"ahler metric has $k$-positive Ricci curvature. Then a generic fiber of the MRC fibration of $M$ has dimension at least $n-k+1$.
\end{theorem}
Since it might be of independent interest, we would also like to point out that the proof of Theorem \ref{mthm_k_ricci} immediately yields the following corollary.
\begin{corollary}
Let $M$ be a projective manifold with a K\"ahler metric with $k$-semi-positive Ricci curvature, where $k\in \{1,\ldots, \dim M\}$. Assume that that there exists at least one point of $M$ at which the K\"ahler metric has $k$-positive Ricci curvature. Let $Z$ be a projective $k$-dimensional manifold with pseudo-effective canonical line bundle. Then there does not exist a dominant rational map from $M$ to $Z$.
\end{corollary}

Our second main theorem is the subsequent one. The invariant $\rr$ captures the largest codimension of maximal subspaces in the tangent spaces on which the holomorphic sectional curvature vanishes. Its analog in the semi-negative curvature case (which we denoted by $r_M$) was used for the structure theorems in \cite{heier_lu_wong_2015} and it is similar to the established notion of the Ricci rank. One of our more philosophical points is that in connection with holomorphic sectional curvature, $r_M$ and $\rr$ are appropriate numerical invariants to consider. It is thus tempting to call them the {\it HSC rank}. For a precise definition, we refer to Section 2.

\begin{theorem} \label{mthm_hsc}
Let $M$ be a projective manifold with a K\"ahler metric of semi-positive holomorphic sectional curvature. Then a generic fiber of the MRC fibration of $M$ has dimension at least $\rr$.
\end{theorem}
\begin{remark}
In the special case that $\rr=\dim M$, the above theorem yields that $M$ is rationally connected. It is immediate from the definition that $\rr=\dim M$ is achieved as soon as the holomorphic sectional curvature is positive at one point of $M$ (and semi-positive on all of $M$). In particular, Theorem \ref{mthm_hsc} yields that a projective manifold with a K\"ahler metric of positive holomorphic sectional curvature is rationally connected. This had been conjectured by S.-T. Yau (even in the K\"ahler case) and was included in his 1982 list of problems \cite[Problem 47]{Yau_120_problems}.\par
Moreover, recall that Tsukamoto \cite{Tsukamoto_1957} proved that a compact K\"ahler manifold of positive holomorphic sectional curvature is simply connected. We remark that the above theorem yields the same conclusion for the case of a projective K\"ahler manifold $M$ of positive holomorphic sectional curvature and in fact extends it to the case of semi-positive holomorphic sectional curvature with $\rr=\dim M$.  The reason is that, due to \cite{Campana_twistor}, it is known that a rationally connected projective manifold is simply connected. Recall furthermore that it is known that a rationally connected projective manifold has no global non-zero covariant holomorphic tensor fields and that the converse of this statement is a conjecture of Mumford (see \cite[p.\ 202]{kollar_book}).  
\end{remark}
The proof of the next theorem is essentially identical to that of Theorem \ref{mthm_hsc}. Thus, we simply record the theorem here and omit its proof.
\begin{theorem}\label{k_scalar_thm}
Let $M$ be a projective manifold with a K\"ahler metric of semi-positive scalar curvature with respect to $k$-dimensional subspaces, where $k\in \{1,\ldots,n:=\dim M\}$. Assume that there exists at least one point of $M$ at which the scalar curvature with respect to $k$-dimensional subspaces is positive. Then a generic fiber of the MRC fibration of $M$ has dimension at least $n-k+1$.
\end{theorem}
In this case, the corresponding statement regarding the non-existence of certain maps is the following corollary.
\begin{corollary}
Let $M$ be a projective manifold with a K\"ahler metric of semi-positive scalar curvature with respect to $k$-dimensional subspaces, where $k\in \{1,\ldots, \dim M\}$. Assume that there exists at least one point of $M$ at which the scalar curvature with respect to $k$-dimensional subspaces is positive. Let $Z$ be a projective $k$-dimensional manifold with pseudo-effective canonical line bundle. Then there does not exist a dominant rational map from $M$ to $Z$.\end{corollary}
At first glance, it might seem that the kinds of methods used in our proof of Theorems \ref{mthm_hsc} and \ref{k_scalar_thm} require that the generic fibers of the dominant rational map in the above corollary are assumed to be compact. However, this is not the case, as one can show with the following additional arguments. Let us assume that a map $\varsigma$ from $M$ to $Z$ exists as in the corollary. We may remove the indeterminacy of $\varsigma$ as usual with a holomorphic birational map $\rho: M^* \to M$ such that $\varsigma\circ \rho:M^*\to Z$ is holomorphic. It is immediate from the definition that the generic fibers of the respective MRC fibrations of $M$ and $M^*$ have the same dimension, which is at least $n-k+1$ according to Theorem \ref{k_scalar_thm}. Since $Z$ is assumed to be $k$-dimensional, the generic fibers of $\varsigma\circ \rho$ are of the strictly smaller dimension $n-k$, and it is now clear that there exists a rational curve through a generic point of $Z$, i.e., $Z$ is uniruled. However, this is impossible due to \cite{BDPP_13} and the assumption that $Z$ has pseudo-effective canonical line bundle.
\begin{remark}
We would like to point out that in our theorems, the Ricci curvature may have negative eigenvalues and our theorems will still apply as long as the assumed positivity conditions hold, while other works in this direction seem to require that the Ricci curvature is semi-definite.
\end{remark}
\begin{remark}
It is clear that all we really require in terms of positivity assumptions is the positivity of the integrals appearing in our proofs. It is for this same reason that our result in \cite{heier_wong_cag} is stated in terms of {\it total} scalar curvature, which is the integral of the scalar curvature function. Therefore, making pointwise positivity assumptions as we do in our theorems is actually overkill and basically just due to our desire to formulate iconic theorems.
\end{remark}
\begin{remark}
In the remainder of this paper, we will use the symbol $k$ as part of the set of indices $i,j,k,l$, and thus speak of $\kappa$-positive Ricci curvature etc., again with $\kappa\in \{1,\ldots,n:=\dim M\}$.
\end{remark}
The outline of this paper is as follows. In Section \ref{diff_alg_geom_backgr}, we provide the basic definitions and technical background. In Section \ref{proof_section}, our two main theorems are proven. The proofs completely coincide except at their very end, where the respective specific positivity assumptions are used. In the subsection entitled Proof of Theorem \ref{mthm_hsc}, we provide the part of the proof that differs from the proof of Theorem \ref{mthm_k_ricci}.

\begin{acknowledgement}
The first author thanks Brian Lehmann for a helpful discussion about pseudo-effective line bundles.
\end{acknowledgement}

\section{Differential and algebraic geometric background material}\label{diff_alg_geom_backgr}
In this section, we will give all the relevant definitions used throughout the paper. They are mostly well-known, but we hope the reader will find it useful to have all of them gathered neatly in one place.
\subsection{Notions of curvature}
Let $M$ be an $n$-dimensional complex manifold. If $V$ is a hermitian vector bundle on $M$ of rank $\rho$, then we denote by $\theta_{\alpha\beta}$ $(\alpha,\beta=1,\ldots, \rho)$ the connection matrix of the metric connection with respect to a local frame $f_1,\ldots,f_\rho$. The corresponding curvature tensor $\Theta$ is determined by $$\Theta_{\alpha\beta}=d\theta_{\alpha\beta}-\sum_{\gamma=1}^\rho\theta_{\alpha\gamma}\wedge\theta_{\gamma\beta}.$$\par
One of the most interesting cases of the above is when $V$ is the holomorphic tangent bundle of $M$ with a local frame $\frac{\partial}{\partial z_1},\ldots,\frac{\partial}{\partial z_n}$, and the hermitian metric is a K\"ahler metric. Let us recall the basics of this case.\par
Let  $z_1,\ldots,z_n$ be local coordinates on $M$. Let
\begin{equation*}
g=\sum_{i,j=1}^n g_{i\bar j} dz_i\otimes d\bar{z}_j
\end{equation*}
be a hermitian metric on $M$ and 
$$\omega=\frac{\sqrt{-1}}{2}\sum_{i,j=1}^n g_{i\bar{j}} dz_i\wedge d\bar{z}_{{j}}$$ the $(1,1)$-form associated to $g$. Then $g$ is called {\it K\"ahler} if and only if locally there exists a real-valued function $f$ such that $g_{i\bar j}=\frac{\partial^2 f}{\partial z_i\partial \bar z_j}$. An equivalent characterization of the K\"ahler property is that $d\omega = 0$. \par
For holomorphic tangent vectors $u=\sum_{i=1}^{n} u_i\frac{\partial}{\partial z_i}$, $v=\sum_{i=1}^{n} v_i\frac{\partial}{\partial z_i}$, we define the $(1,1)$-form $\Theta_{u\bar v}$ to be
$$\Theta_{u\bar v}=\sum_{i,j,k=1}^{n}\Theta_{ik}g_{k\bar j}u_i \bar v _j.$$
Moreover, the curvature $4$-tensor is given by
$$R(\frac{\partial}{\partial z_i},\frac{\partial}{\partial \bar z_j},\frac{\partial}{\partial z_k},\frac{\partial}{\partial \bar z_l})=\Theta_{\frac{\partial}{\partial z_k}\frac{\partial}{\partial \bar z_l}}(\frac{\partial}{\partial z_i},\frac{\partial}{\partial \bar z_j}).$$
Since we assume the metric $g$ to be K\"ahler, the curvature $4$-tensor satisfies the K\"ahler symmetry
$$R(\frac{\partial}{\partial z_i},\frac{\partial}{\partial \bar z_j},\frac{\partial}{\partial z_k},\frac{\partial}{\partial \bar z_l})=R(\frac{\partial}{\partial z_k},\frac{\partial}{\partial \bar z_l},\frac{\partial}{\partial z_i},\frac{\partial}{\partial \bar z_j}).$$
Now, we define the {\it Ricci curvature form} to be the $(1,1)$-form
$$Ric= \sqrt{-1}Tr (\Theta)(\cdot,\cdot)=\sqrt{-1}\sum _{i,j,k=1}^n R(\frac{\partial}{\partial z_i},\frac{\partial}{\partial \bar z_j},e_k,\bar e_k) dz_i\wedge d\bar{z}_{{j}},$$
where $e_1,\ldots,e_n$ is a unitary frame. If we let 
$$R_{i\bar j}= \sum _{k=1}^n R(\frac{\partial}{\partial z_i},\frac{\partial}{\partial \bar z_j},e_k,\bar e_k),$$
then
$$Ric = \sqrt{-1} \sum_{i,j=1}^{n} R_{i\bar j} dz_i\wedge d\bar{z}_{{j}}.$$\par
We say that the Ricci curvature is {\it $\kappa$-(semi-)positive} at the point $p\in M$ if the eigenvalues of the hermitian $n \times n$ matrix $R_{i\bar j}$ at $p$ have the property that any sum of $\kappa$ of them is (semi-)positive. Note that if $\kappa \leq \kappa '$, then being $\kappa$-(semi-)positive implies being $\kappa'$-(semi-)positive. Moreover, by definition, being $n$-(semi-)positive is the same as having (semi-)positive scalar curvature, and being $1$-(semi-)positive is the same as having (semi-)positive Ricci curvature. We say that the Ricci curvature is {\it $\kappa$-(semi-)positive} if it is $\kappa$-(semi-)positive at all points $p\in M$.\par
The {\it scalar curvature with respect to a $\kappa$-dimensional subspace} $\Sigma\subset T_p M$ is defined to be
$$\sum_{i,j=1}^\kappa R(e_i,\bar e_i, e_j,\bar e_j),$$
where $e_1,\ldots,e_\kappa$ is a unitary basis for $\Sigma$. When $\Sigma=T_p M$, we simply speak of the {\it scalar curvature}. We say that the {\it scalar curvature with respect to $\kappa$-dimensional subspaces} is {\it (semi-)positive at the point $p\in M$} if the scalar curvature with respect to all $\kappa$-dimensional subspaces $\Sigma\subset T_p M$ is \mbox{(semi-)}positive. We say that the scalar curvature with respect to $\kappa$-dimen\-sional subspaces is {\it (semi-)positive} if it is (semi-)positive at all points $p\in M$. Note again that these (semi-)positivity properties are preserved under increasing the value of $\kappa$.\par

If $\xi=\sum_{i=1}^n\xi_i \frac{\partial }{\partial z_i}$ is a non-zero complex tangent vector at $p\in M$, then the {\it holomorphic sectional curvature} $H(\xi)$ is given by
\begin{equation*}
H(\xi)=\left( 2 \sum_{i,j,k,l=1}^n R_{i\bar j k \bar l}(p)\xi_i\bar\xi_j\xi_k\bar \xi_l\right) / \left(\sum_{i,j,k,l=1}^ng_{i\bar j}g_{k\bar l} \xi_i\bar\xi_j\xi_k\bar \xi_l\right).
\end{equation*}
We say that the holomorphic sectional curvature is {\it (semi-)positive} if 
$$H(\xi) > 0\ (\geq 0) \quad \forall \xi\quad \forall p\in M.$$ 
Now, let us assume that $H$ is semi-positive. We define the invariant $\rr$ as follows. For $p\in M$, let $\mathfrak{n}(p)$ be the maximum of those integers $\ell\in \{0,\ldots,\dim M\}$ such that there exists a $\ell$-dimensional subspace $L\subset T_p M$ with $H(\xi)=0$ for all $\xi\in L\backslash \{\vec{0}\}$. Set $\rr:=n- \min_{p\in M} \mathfrak{n}(p).$ Note that by definition $\rr=0$ if and only if $H$ vanishes identically. Also, $\rr=\dim M$ if and only there exists at least one point $p\in M$ such that $H$ is positive at $p$. Moreover, $\mathfrak{n}(p)$ is upper-semicontinuous as a function of $p$, and consequently the set
$$\{p\in M\ |\ n-\mathfrak{n}(p)=\rr\}$$
is an open set in $M$ (in the classical topology). \par

We conclude this subsection with some hopefully useful historical remarks, in particular about the relationship of Ricci and holomorphic sectional curvature. \par
The existence of a K\"ahler metric of negative holomorphic sectional curvature implies the existence of a (different) metric of negative Ricci curvature, as was very recently established by \cite{wu_yau} in the projective case and, in the general case, by \cite{tosatti_yang}. Previously, the three-dimensional projective case of this statement had been proven in \cite{heier_lu_wong_mrl}. The paper \cite{heier_lu_wong_2015} contains partial positivity results for the canonical line bundle in the semi-negative projective case. Conversely, it is easy to show via a Schwarz Lemma that there are many K\"ahler manifolds of negative Ricci curvature which do not admit metrics of negative holomorphic sectional curvature. \par
Under the assumption of positive scalar curvature on a hermitian manifold, in the pioneering paper \cite{Yau_74}, S.-T. Yau proved that the Kodaira dimension is negative. At the other end of the positive curvature spectrum, Siu and Yau \cite{siu_yau} proved the Frankel conjecture, which states that a compact K\"ahler manifold of positive bisectional curvature is biholomorphic to projective space. At approximately the same time, Mori \cite{Mori_79} established a more general conjecture of Hartshorne, which states that a
compact complex manifold with ample tangent bundle is biholomorphic to projective space. \par

In light of the above, it may come as a bit of a surprise that positive holomorphic sectional curvature does in general not imply the existence of a metric of positive Ricci curvature, as pointed out by Hitchin \cite{Hitchin} in his discussion of the Hirzebruch surfaces $\PP(\OO_{\PP^1}(a) \oplus  \OO_{\PP^1})$, $a\in \{0,1,2,\ldots\}$. Moreover, we do not know if a compact K\"ahler manifold of positive holomorphic sectional curvature is projective. Conversely, the question of the implications of positive Ricci curvature for the existence of a metric of positive holomorphic sectional curvature seems to be open at this point as well. 

\subsection{The MRC fibration}
On a smooth projective variety $X$ it is quite natural to consider the equivalence relation of being connected by a rational curve, i.e., two points $x,y\in X$ are considered to be equivalent if and only if there exists a rational curve containing both $x$ and $y$. The problem is that the map to the quotient under this equivalence relation is in general not a good map. This is the case, for example, when $X$ is a very general projective K3 surface because such an $X$ possesses countably infinitely many rational curves. The question of how to obtain a good map based on this equivalence relation has been answered by Campana \cite{campana_con_rat} and Koll\'ar-Miyaoka-Mori \cite{k_m_m_jag}. The following theorem is \cite[Theorem 2.7]{k_m_m_jag}.
\begin{theorem}
Let $X$ be a smooth proper algebraic variety over an algebraically closed field of characteristic $0$. Then there exist an open dense subset $U\subset X$ and a proper smooth morphism $f:U\to Z$ with the following properties:
\begin{itemize}
\item Every fiber of $f$ is rationally connected.
\item For a sufficiently general $z\in Z$, there are no rational curves $D\subset X$ such that $\dim D\cap f^{-1}(z) = 0$. 
\end{itemize}
\end{theorem}
The morphism $f$ is called the {\it maximally rationally connected fibration} of $X$ (or {\it MRC fibration} for short). \par
Equivalently, one can think of $f$ as being a dominant rational map which
is holomorphic and proper on a dense open subset of $X$ (aka ``almost
holomorphic"). Its general fiber is rationally connected. Moreover, the fibers of $f$ are maximal in the sense that $f$ factors
through any other map with rationally connected fibers.\par
Furthermore, we may assume that the base $Z$ is smooth, because if it is not, we
can compose $f$ with a birational map $Z \dashrightarrow Z'$ which resolves the
singularities of $Z$. The rational map $X \dashrightarrow Z'$ still has all the properties of
the original $f$. In general, it is clear that the MRC fibration is well-defined only up to birational equivalence.\par
Finally, it will be central to our argument that the base $Z$ of the MRC fibration is not uniruled. This statement can be obtained as an immediate corollary (\cite[Corollary 1.3]{GHS}) to the following theorem of Graber-Harris-Starr (\cite[Theorem 1.1]{GHS}).\begin{theorem}
Let $f:X\to B$ be a proper morphism of complex varieties with $B$ a smooth curve. If the general fiber of $f$ is rationally connected, then $f$ has a section.
\end{theorem}
It is then a consequence of \cite{BDPP_13} that the canonical line bundle $K_Z$ of the base $Z$ is pseudo-effective.

\section{Proof of Theorems \ref{mthm_k_ricci} and \ref{mthm_hsc}}\label{proof_section}
In this section, we prove Theorems \ref{mthm_k_ricci} and \ref{mthm_hsc}. The proofs are initially the same, but there are some subtle differences towards the end in how the contradiction is obtained. In particular,  the use of the positivity condition in the proof of Theorem \ref{mthm_k_ricci} is a matter of linear algebra, whereas the proof of Theorem \ref{mthm_hsc} requires a lemma of Berger and a global argument involving the Divergence Theorem.

\subsection{Proof of Theorem \ref{mthm_k_ricci}} 
We start with some general observations about MRC fibrations for which we could not find a reference in the literature. We thus hope that this part is of independent interest for the overall understanding of this important map.\par

Let $V\subset M$ denote the indeterminacy locus of our MRC fibration $f$. On $M\backslash V$, the pullback of the tangent bundle of $Z$, denoted $f^*TZ$, is a rank $m:=\dim Z$ holomorphic vector bundle. Since the codimension of $V$ is at least two, this vector bundle can be extended to all of $M$ as a reflexive sheaf in a unique way, and we denote this extension with the symbol $E$. The canonical line bundle $K_Z$ and its dual $K_Z^*$ can be pulled back under $f$ and extended to all of $M$ as line bundles. We denote these extensions by $f^*K_Z$ and $f^*K_Z^*$ and note that they agree with $\det E^*$ and $\det E$, respectively. \par 
On $M\backslash V$, there is an exact sequence of coherent sheaves
$$ 0\to T_{(M\backslash V/f(M\backslash V))} \to TM \stackrel{df}{\to} f^*TZ\to N\to 0,$$
where $N$ is a coherent sheaf supported on the locus $B$ where $f: M\backslash V\to f(M\backslash V)$ is not smooth (see \cite[Definition 3.4.5]{Sernesi_book}). Let $W=V\cup B$.\par
We will now prove that $$\int_{M} c_1(E) \wedge \omega^{n-1}$$
is non-positive, based on the following proposition. The proof of our theorems will then be finished by establishing that this integral can also be shown to be positive under the assumption that generic fibers of the MRC fibration are of dimension no greater than $n-\kappa$ and, respectively,\ that these fibers are of dimension no greater than $\rr-1$. We suspect the statement of the proposition is known to experts, but for lack of a suitable reference, we provide a proof.
\begin{proposition}\label{pullback_pseff}
Let $\nu:X\dashrightarrow Y$ be a dominant rational map between complex projective manifolds $X,Y$. Let $L$ be a pseudo-effective line bundle on $Y$. Then the pull-back line bundle $\nu^*L$ is a pseudo-effective line bundle on $X$.
\end{proposition}
\begin{proof}
Among the various equivalent definitions of pseudo-effectivity for a line bundle (see \cite{Dem}), a convenient one to use in this context is the following. The line bundle $L$ is pseudo-effective if its numerical class lies in the pseudo-effective cone $\PEff(Y)$, i.e., in the closure of the convex cone generated by Chern classes of effective line bundles in the real N\'eron-Severi vector space $N^1(X)_\RR$.\par
Furthermore, we observe that for an arbitrary holomorphic map $\nu':X'\to Y$ from a complex projective manifold $X'$, the pullback map $\nu'^*:N^1(Y)_\RR\to N^1(X')_\RR$ is an injective linear map of finite dimensional real vector spaces. Moreover, since $\PEff(Y)$ is generated by the classes of line bundles with a section, there is an induced injective map $\nu'^*:\PEff(Y)\to \PEff(X')$. \par
Now, let $\sigma: X'\to X$ be a birational holomorphic map such that the composition $\nu':=\nu\circ \sigma :X'\to Y$ is a holomorphic map. Due to the above remarks, $\nu'^* L$ is a pseudo-effective line bundle on $X'$. Furthermore, the pushforward $\sigma_*$ of divisors induces a linear map $\sigma_*:N^1(X')_\RR\to N^1(X)_\RR$. Since the pushforward of an effective divisor is still effective (or zero), we get an induced map $\sigma_*: \PEff(X')\to \PEff(X)$. Since $\nu^*=\sigma_* \circ \nu'^*$, our claim is proven.
\end{proof}
In \cite[Section 2]{heier_wong_cag}, a linear algebra argument was given for the fact that any pseudo-effective line bundle $P$ on $M$ satisfies
$$ \int_M c_1(P) \wedge \omega ^{n-1} \geq 0.$$
Alternatively, the non-negativity of this integral can be justified by arguing that it holds if the numerical class of $P$ is a positive scalar multiple of the numerical class of an effective line bundle. The non-negativity is then preserved when taking limit.\par
Since 
$$\det E = f^*K_Z^*,$$
we have $c_1(E)=-c_1(f^*K_Z)$. Also, by Proposition \ref{pullback_pseff}, the pseudo-effectivity of $K_Z$ implies the pseudo-effectivity of $f^*K_Z$, and we obtain the desired inequality as follows.
\begin{align}
~& \int_M c_1(E) \wedge \omega ^{n-1}\nonumber\\
=& \int_M c_1(f^*K_Z^*) \wedge \omega ^{n-1}\nonumber\\
=& -\int_M c_1(f^*K_Z) \wedge \omega ^{n-1}\nonumber\\
\leq &\ 0.\nonumber
\end{align}\par
In the rest of this section, we shall infer that the value of the above integral is positive under the assumptions of Theorems \ref{mthm_k_ricci} and \ref{mthm_hsc}, respectively, resulting in contradictions. Our argument is based on the well-known fact that, as a quotient bundle of $TM$ over $M\backslash W$, the vector bundle $E|_{M\backslash W}$ carries an induced hermitian metric $h$ whose curvature is equal to or more positive than that of $g$ on $TM$ over $M\backslash W$.\par
To be precise, let us recall the following standard setup. Locally on $M\backslash W$, we choose a unitary frame $e_1,\ldots,e_n$ for $TM$ such that $e_1,\ldots,e_{n-m}$ is a frame for $S:=T_{(M\backslash V/f(M\backslash V))}$.  The connection matrix for the metric connection on $TM$ is
$$\theta_{TM}= 
\left(\begin{matrix}
                \theta_S&\bar A^T\\
                A& \theta_E
\end{matrix}\right),
$$ 
where $\theta_S,\theta_E$ are the respective connection matrices for $S$ and $E$, and $A\in A^{1,0}(\Hom(S,E))$ is the second fundamental form of $S$ in $TM$. Now, the curvature matrix in terms of the unitary frame $e_1,\ldots,e_n$ is 
$$ \Theta_{TM}= 
\left(\begin{matrix}
                d\theta_S-\theta_S\wedge \theta_S -\bar A^T\wedge A & * \\
                *&d\theta_E-\theta_E\wedge \theta_E-A\wedge \bar A^T 
\end{matrix}\right),
$$
which implies that $\Theta_{E} = \Theta_{TM}|_E + A\wedge \bar A^T$ and, in particular,
\begin{equation}\label{curv_incr}
\Theta_{E} \geq  \Theta_{TM}|_E.
\end{equation}
We let the $(1,1)$-form $\eta$ be the trace of the matrix $\Theta_{E}$ over $M\backslash W$, i.e.,
$$\eta = \sum_{\alpha=1}^m (\Theta_{E})_{\alpha\alpha}.$$
Let $\tilde h$ be an arbitrary smooth hermitian metric on $\det E$ over the entire $M$ with curvature form $\tilde \eta$. If $\tau$ is a local nowhere zero holomorphic section of $\det E$, the ratio
$$q=\frac{(\det h)(\tau,\tau)}{\tilde h(\tau,\tau)}$$
is independent of the choice of $\tau$ and constitutes a smooth positive function on $M\backslash W$. Over that same set, we have the following relationship:
$$\tilde \eta = \eta + \partial\bar\partial \log q.$$\par
By standard techniques from the theory of resolution of singularities and the removal of indeterminacy, there is a compact complex manifold $M^*$ and a surjective holomorphic map $\rho:M^*\to M$ such that $$\rho|_{M^*\backslash \rho^{-1} (W)}:M^*\backslash \rho^{-1} (W) \to M\backslash W$$
is biholomorphic, the total transform $\rho^{*}(W)$ is a divisor with simple normal crossing support in $M^*$, and $f\circ \rho$ is holomorphic. With positive integers $a_i,b_j$, write
$$\rho^{*}(W) = \sum_{i\in I}  a_i D^{(1)}_i + \sum_{j\in J} b_j D^{(2)}_j,$$
where the $D^{(1)}_i$ are the irreducible components of $\rho^{*}(W)$ such that $\rho(D^{(1)}_i)$ has codimension one and the $D^{(2)}_j$ are the irreducible components of $\rho^{*}(W)$ such that $\rho(D^{(2)}_j)$ has codimension at least two.
\par
Furthermore,
\begin{align}
\int_{M} c_1(E) \wedge \omega^{n-1}=&\int_{M}\frac{\sqrt{-1}}{2\pi} \tilde \eta \wedge \omega^{n-1}\nonumber \\ 
=& \int_{M\backslash W}\frac{\sqrt{-1}}{2\pi}\tilde  \eta \wedge \omega^{n-1}\label{remove_W}\\
=& \int_{M\backslash W}\frac{\sqrt{-1}}{2\pi}\eta \wedge \omega^{n-1}+ \int_{M\backslash W}\frac{\sqrt{-1}}{2\pi}\partial\bar\partial \log q  \wedge \omega^{n-1} .\label{split_key_integral_in_two}
\end{align}
The equality \eqref{remove_W} holds because $\tilde \eta$ is a smooth $(1,1)$-form and the removal of a proper closed subset does not affect the value of the integral.\par
We proceed by showing that the second summand in \eqref{split_key_integral_in_two}, i.e.,
\begin{align}
\int_{M\backslash W}\frac{\sqrt{-1}}{2\pi}\partial\bar\partial \log q  \wedge \omega^{n-1}=& \int_{M^*\backslash \rho^{-1} (W)}\frac{\sqrt{-1}}{2\pi}\partial\bar\partial \log \rho^* q  \wedge \rho^* \omega^{n-1},\nonumber
\end{align}
is a non-negative number. To be more precise, we will see that this integral is always non-negative and, additionally, positive if and only if there exists a divisor along which $f$ is not smooth. The reason is the following chain of equalities.
\begin{align}
 &\ \int_{M^*\backslash \rho^{-1} (W)}\frac{\sqrt{-1}}{2\pi}\partial\bar\partial \log \rho^* q  \wedge \rho^* \omega^{n-1}\nonumber\\ =&\ \sum_{i\in I} a_i \int_{D^{(1)}_i} \rho^* \omega^{n-1} +\sum_{j\in J} b_j \int_{D^{(2)}_j} \rho^* \omega^{n-1} \label{PL} \\
 =&\ \sum_{i\in I} a_i \int_{\rho(D^{(1)}_i)} \omega^{n-1} +\sum_{j\in J} b_j \int_{\rho(D^{(2)}_j)} \omega^{n-1} \nonumber\\
 =&\ \sum_{i\in I} a_i \int_{\rho(D^{(1)}_i)} \omega^{n-1} + 0 \label{codim_2_gives_0} 
\end{align}
Note that equality \eqref{PL} is due to the Poincar\'e-Lelong equation and equality \eqref{codim_2_gives_0} is due to the fact that $\dim \rho(D^{(2)}_j) \leq n-2$. Moreover, if $I=\emptyset$, then the value of \eqref{codim_2_gives_0} is zero. If $I\not =\emptyset$, then the value of \eqref{codim_2_gives_0} is positive.\par
We now observe that on $M\backslash W$:
\begin{align}&\ \sqrt{-1}\eta\wedge \omega^{n-1} \nonumber\\ =&\ \sqrt{-1} \sum_{\alpha=1}^m (\Theta_{E})_{\alpha\alpha}\wedge \omega^{n-1}\nonumber\\
\geq &\ \sqrt{-1} \sum_{\alpha=1}^m (\Theta_{TM})_{(n-m+\alpha)\, (n-m+\alpha)}\wedge \omega^{n-1}\label{curv_incr_invoke}\\
= &\ \sqrt{-1} \sum_{\alpha=1}^m (\Theta_{TM})_{e_{(n-m+\alpha)}\,\bar e_{(n-m+\alpha)}}\wedge \omega^{n-1}\nonumber\\
=&\ \sqrt{-1} \sum_{\alpha=1}^m\sum_{i,j=1}^n R(\frac{\partial}{\partial z_i},\frac{\partial}{\partial \bar z_j},e_{(n-m+\alpha)},\bar e_{(n-m+\alpha)}) dz_i\wedge d\bar{z}_{{j}}\wedge \omega^{n-1}\nonumber\\
=&\ \sqrt{-1} \sum_{\alpha=1}^m \sum_{i,j=1}^n R(e_{(n-m+\alpha)},\bar e_{(n-m+\alpha)},\frac{\partial}{\partial z_i},\frac{\partial}{\partial \bar z_j}) dz_i\wedge d\bar{z}_{{j}}\wedge \omega^{n-1}\label{kaehler_symm}\\
=&\ \frac 2 n \left(\sum_{\alpha=1}^m Ric(e_{(n-m+\alpha)},\bar e_{(n-m+\alpha)})\right) \omega^{n}\label{application_of_trace_formula}
\end{align}
The inequality in \eqref{curv_incr_invoke} is due to the curvature increasing property \eqref{curv_incr}. The equality \eqref{kaehler_symm} is due to the K\"ahler symmetry of the curvature $4$-tensor, and equality \eqref{application_of_trace_formula} is due to the trace formula. \par

To obtain the desired contradiction, let us assume that a general fiber of the MRC fibration is of dimension no greater than $n-\kappa$. This condition is clearly equivalent to $m \geq \kappa$. The technical reason behind our argument is the theory of minimax formulae and extremal partial traces for the eigenvalues of hermitian matrices. For a nice account of this material we refer to \cite[Section 1.3.2]{Tao_book}. In a nutshell, the key point is the following.\par
Given an $n\times n$ matrix $T$ and an $m$-dimensional subspace $S\subset \CC^n$, one can define the {\it partial trace} of $T$ with respect to $S$ and a fixed hermitian inner product to be the expression 
$$\tr_S T := \sum_{i=1}^m v_i^*Tv_i,$$
where $v_1,\ldots,v_m$ is any unitary basis of $S$. We simply write $\tr T$ for $\tr_{\CC^n} T$. The displayed expression is easily seen to be independent of the choice of the unitary basis and thus well-defined. If we assume that $T$ is hermitian and let $\lambda_1\geq\ldots\geq \lambda_n$ be the eigenvalues of $T$, then for any $1\leq m\leq n$, one has
$$\lambda_1+\ldots+\lambda_m = \sup_{\dim(S)=m} \tr_S T$$
and
\begin{equation}\lambda_{n-m+1}+\ldots+\lambda_n = \inf_{\dim(S)=m} \tr_S T.\label{inf_formula}\end{equation}\par
Now, the sum $\sum_{\alpha=1}^m Ric(e_{(n-m+\alpha)},\bar e_{(n-m+\alpha)})$ is the partial trace of the hermitian $n \times n$ matrix $R_{i\bar j}$ with respect to $S=\spann \{e_{(n-m+1)},\ldots, e_{n}\}$. Thus, according to \eqref{inf_formula}, the expression \eqref{application_of_trace_formula} is bounded below by
\begin{align}
&\ \frac 2 n \left(\sum_{\alpha=1}^m \tau_{n-m+\alpha}\right)\omega^{n}\label{extremal_partial_trace},
\end{align}
where $\tau_{n-m+1},\ldots,\tau_{n}$ denote the $m$ smallest eigenvalues of the hermitian $n \times n$ matrix $R_{i\bar j}$. Due to $m \geq \kappa$ and the assumption of $\kappa$-positivity, $\sum_{\alpha=1}^m \tau_{n-m+\alpha}$ is positive everywhere on $M$, and the expression \eqref{extremal_partial_trace} is therefore bounded below by $\varepsilon \omega^{n}$ for some positive number $\varepsilon$.
To conclude the proof, we observe that altogether
\begin{align}
~&\ \int_M c_1(E) \wedge \omega ^{n-1}\nonumber\\
\geq &\ \int_{M\backslash W} \frac{\sqrt{-1}}{2\pi}\eta\wedge \omega^{n-1}\label{due_to_nonneg_of_summand}\\
\geq &\ \frac{1}{2\pi}  \int_{M\backslash W} \varepsilon \omega ^{n}\nonumber\\
> &\ 0\nonumber.
\end{align}
Note that the inequality \eqref{due_to_nonneg_of_summand} is due to the non-negativity of the second summand in \eqref{split_key_integral_in_two}.
We have obtained the desired contradiction to complete the proof of Theorem \ref{mthm_k_ricci}.
\subsection{Proof of Theorem \ref{mthm_hsc}}Let us assume that a general fiber of the MRC fibration is of dimension no greater than  $\rr-1$. This condition is clearly equivalent to $m \geq n-\rr+1$. We enter the proof of Theorem \ref{mthm_k_ricci} at the point where it was determined that it remained to prove the positivity of 
\begin{align}
\int_{M\backslash W} \frac{\sqrt{-1}}{2\pi}\eta\wedge \omega^{n-1}\label{key_integral}
\end{align}
in order to obtain a contradiction. \par
Now, we recall that the set $W$ is the union of the indeterminacy locus $V$ of the MRC fibration $f$ and the locus where the map does not have full rank. Since $f$ is almost holomorphic and due to the standard generic smoothness property of holomorphic maps, the set $W$ does not intersect a generic fiber of $f$. Thus, there is a dense Zariski-open subset $Z'\subset Z$ such that 
\begin{align} f|_{f^{-1}(Z')}:f^{-1}(Z')\to Z' \nonumber\end{align}
is a holomorphic submersion. Moreover, since the integrand in \eqref{key_integral} is a smooth form, the domain of integration can be replaced with the dense subset $f^{-1}(Z')$. Thus, it suffices to take an arbitrary but fixed small open set $\VV\subset Z'$ and prove that 
\begin{align}
\int_{f^{-1}(\VV)} \frac{\sqrt{-1}}{2\pi}\eta\wedge \omega^{n-1}\label{key_integral_over_inverse_of_small_open}
\end{align}
is non-negative in general and positive for a certain choice of $\VV$. \par

Let $w_1,\ldots,w_m$ be local coordinates on a given $\VV$. For each $q\in \VV$, we write $M_q:=f^{-1}(q)$ for the fiber over $q$. For each $p\in f^{-1}(\VV)$, we can choose tangent vectors
$$\left\{\widetilde{ \frac \partial {\partial w_1}},\ldots,\widetilde{ \frac \partial {\partial w_m}}\right\}$$
forming a smooth frame for $((T_{p}M_{f(p)})^\perp)_{p\in f^{-1}(\VV)}$ and having the property that 
$$df_p\left(\widetilde{ \frac \partial {\partial w_i}}\right) =  \frac \partial {\partial w_i} \quad (i=1,\ldots,m).$$\par
In a small neighborhood $U$ of an arbitrary but fixed $p\in M_q$, we can choose holomorphic coordinate functions $z_1,\ldots,z_{n-m}$ such that for all $p'\in U$:
$$T_{p'}M_{f(p')}=\spann \left\{\frac \partial {\partial z_1},\ldots,\frac \partial {\partial z_{n-m}}\right\}.$$
We write 
$$ \phi_{1}=\frac \partial {\partial z_1}^*, \ldots, \phi_{n-m}=\frac \partial {\partial z_{n-m}}^*,\phi_{n-m+1}=\widetilde{ \frac \partial {\partial w_1}}^*,\ldots,\phi_{n}=\widetilde{ \frac \partial {\partial w_m}}^*$$ for the dual fields of one-forms.\par
Let $h$ continue to denote the induced hermitian metric on $E|_{M\backslash W}$. Locally, we write
$$\eta = \sum_{i,j=1}^n \eta_{i\bar j} \phi_i\wedge \bar{\phi}_j.$$
If we denote by $s_E$ the trace of $\eta$ with respect to $g$, we can rewrite \eqref{key_integral_over_inverse_of_small_open} by the trace formula as
\begin{align}
\int_{f^{-1}(\VV)} \frac{\sqrt{-1}}{2\pi}\eta\wedge \omega^{n-1}= \int_{f^{-1}(\VV)} \frac{1}{n \pi} s_E\omega^{n}.\nonumber
\end{align}\par
Returning to the above-described general linear algebra setting for a moment, observe that for any subspace $S$ and its orthogonal complement $S^\perp$, the trace $\tr T$ satisfies
$$\tr T = \tr_{S^\perp} T + \tr_{S} T.$$\par
If we apply this with $S=T_{p'}M_{f(p')}$, then we obtain
$$s_E=K_1+K_2,$$
where 
$$K_2=\sum_{i,j=1}^{n-m} g^{i\bar j}\eta_{i\bar j}= - \sum_{i,j=1}^{n-m} g^{i\bar j}\frac{\partial ^2 \log \det h}{\partial z_i\partial\bar{z}_j},$$
and $K_1$ is the partial trace of $\eta$ with respect to $S^\perp$ and $g$. Due to the curvature increasing property, $K_1$ is bounded from below by the scalar curvature of $g$ with respect to the plane $S^\perp$, which we refer to as $\hat K_1$. Due to \cite[Lemme 7.4]{Berger}, $\hat K_1$ can be expressed as a positive constant times the integral of $H$ over the unit sphere of vectors in $S^\perp$. Due to the overall semi-positivity of $H$, we can conclude $\hat K_1 \geq 0$. Furthermore, at a point $p$ with $n-\mathfrak{n}(p)=\rr$, due to 
$$\dim S^\perp=m \geq n-\rr+1=\mathfrak{n}(p)+1,$$
we actually have $\hat K_1(p) > 0$. Note that due to the openness of the set $\{p\in M\ |\ n-\mathfrak{n}(p)=\rr\}$, there exists a $\VV$ with 
$$\{p\in M\ |\ n-\mathfrak{n}(p)=\rr\} \cap f^{-1}(\VV) \not = \emptyset.$$
We can thus conclude $$\int_{f^{-1}(\VV)}K_1\omega^n \geq \int_{f^{-1}(\VV)}\hat K_1\omega^n > 0,$$ 
and it remains to show that the integral $\int_{f^{-1}(\VV)}K_2\omega^n$ is non-negative. A direct computation yields (we write $\hbar$ for $\det h$)
\begin{align}
K_2=& - \sum_{i,j=1}^{n-m} g^{i\bar j}\frac{\partial ^2 \log \hbar}{\partial z_i\partial\bar{z}_j}\nonumber\\
=& - \sum_{i,j=1}^{n-m} \frac{g^{i\bar j}}{\hbar} \frac{\partial ^2 \hbar}{\partial z_i\partial\bar{z}_j}+\sum_{i,j=1}^{n-m} \frac{g^{i\bar j}}{\hbar^2} \frac{\partial  \hbar}{\partial z_i}\frac{\partial  \hbar}{\partial\bar{z}_j}\nonumber\\
=& \frac{\Delta' \hbar}{\hbar} + \sum_{i,j=1}^{n-m} \frac{g^{i\bar j}}{\hbar^2} \frac{\partial \hbar}{\partial z_i}\frac{\partial \hbar}{\partial\bar{z}_j},\label{pos_part}
\end{align}
where $\Delta' $ is the Laplacian operator on the fibers with respect to the restriction of the metric $g$ to the fibers. Since the second summand in \eqref{pos_part} is always non-negative, we are done if we can prove that
$$\int_{f^{-1}(\VV)} \frac{\Delta' \hbar}{\hbar}  \omega^n = 0.$$\par
In order to do this, note that the normal bundle of a general fiber is the trivial bundle of rank $m$. Therefore, we can regard $\hbar$ as global smooth function of any given general fiber. We can rewrite 
$$\frac{\omega^n}{\hbar}=\omega' \wedge\phi_{n-m+1}\wedge\bar \phi_{n-m+1}\wedge\ldots  \wedge\phi_{n}\wedge\bar \phi_{n},$$
where $\omega'$ is the volume form of the restriction of the metric $g$ on the fibers. Applying the Fubini Theorem of iterated integrals, we have
$$\int_{f^{-1}(\VV)} \frac{\Delta' \hbar}{\hbar}  \omega^n = \int_\VV\left( \int_{f^{-1}(w)} (\Delta' \hbar) \omega'\right)dw _1\wedge d\bar w _1\wedge\ldots\wedge dw _m\wedge d\bar w _m.$$
It follows from the Divergence Theorem that on the compact manifold $f^{-1}(w)$:
$$ \int_{f^{-1}(w)} (\Delta' \hbar) \omega' =0$$
for all $w\in \VV$. This finishes the proof.

\end{document}